\documentclass[english,reqno]{amsart}
\usepackage[utf8]{inputenc}
\usepackage{babel,csquotes}
\usepackage[left=1in,right=1in,top=1in,bottom=1in]{geometry}
\usepackage{amssymb}
\newtheorem*{thm}{Theorem}
\newtheorem{fact}{Fact}

\begin{document}
\allowdisplaybreaks[4]

\title{A note on the Triple Product Property subgroup capacity of finite groups}
\author{Ivo Hedtke}
\address{Institute of Computer Science, University of Halle-Wittenberg, D-06099 Halle, Germany}
\email{hedtke@informatik.uni-halle.de}
\begin{abstract}
In the context of group-theoretic fast matrix multiplication the TPP capacity is used to bound the exponent $\omega$ of matrix multiplication. We prove a new and sharper upper bound
for the TPP subgroup capacity of a finite group.
\end{abstract}
\maketitle

\section{Introduction}\thispagestyle{empty}

\noindent In the context of group-theoretic fast matrix multiplication (see \cite{Cohn2003} for an introduction) the \emph{TPP subgroup capacity} is used to bound the exponent $\omega$ of matrix multiplication. New upper bounds for the TPP subgroup capacity can be used to identify groups that do not lead to a nontrivial upper bound for $\omega$ via subgroup \emph{TPP triples}.
Our new bound \eqref{eq:NewBound} also gives a hint why \textsc{Cohn} and \textsc{Umans} state that \enquote{nonabelian simple groups appear to be a fruitful source of groups} with a large TPP capacity: The TPP capacity of abelian groups is trivial and the normal core of simple groups is equal to $1$.
Some researchers believe that triples of subgroups will never lead to a new nontrivial upper bound for $\omega$ in the context of the (TPP). Maybe this new bound will help to prove or disprove this conjecture.

With $Q(X):=\{xy^{-1}:x,y\in X\}$ we denote the \emph{right quotient} of $X$. A triple $(S,T,U)$ of subsets $S,T,U\subseteq G$ of a group $G$ fulfills the so-called \emph{Triple Product Property} (TPP) if for $s\in Q(S)$, $t\in Q(T)$ and $u \in Q(U)$, $stu=1$ holds iff $s=t=u=1$.
In this case we call $(S,T,U)$ a \emph{TPP triple} of $G$. The \emph{TPP capacity}
$\beta(G) := \max \big\{|S|\cdot|T|\cdot |U| : (S,T,U) \text{ is a TPP triple of }G\big\}$
is the biggest size of a TPP triple of $G$. We can use
\begin{gather}\label{omega}
\beta(G)^{\omega/3}\leq D_\omega(G)
\end{gather}
to find new nontrivial upper bound for the exponent
$
\omega := \inf \{r \in \mathbb R  :  M(n)=\mathcal O(n^r)\}
$
of matrix multiplication, where $D_\omega(G) := \sum d_i^\omega$ and $\{d_i\}$ are the character degrees of $G$. Here $M(n)$ denotes the number of field operations in characteristic $0$ required to multiply two $(n\times n)$ matrices. The \emph{TPP subgroup capacity} $\beta_\mathrm{g}$ is defined like $\beta$, but we restrict $S$, $T$ and $U$ to be subgroups of $G$. Note that $\beta_\mathrm{g} \leq \beta$ holds. Therefore, $\beta_\mathrm{g}$ can be used in the same way like $\beta$ to bound $\omega$, but the result is not as strong as with the TPP capacity $\beta$. On the other hand it is easier to deal with subgroups instead of subsets, especially in (brute-force) search algorithms (see \cite{HedtkeMurthy2011} for details). Note that $\beta_\mathrm{g} (G) \geq |G|$, because $(G,1,1)$ is a TPP triple for every group $G$.

\begin{fact}\label{fact:order}
\cite[Lem.~2.1]{Cohn2003}
Without loss of generality we can assume that $|S|\geq |T| \geq |U|$.
\end{fact}

\begin{fact}\label{fact:neumann}
\cite{Neumann2011}
If $(S,T,U)$ is a TPP triple with $|S|\geq |T| \geq |U|$, then $|S|(|T|+|U|-1)\leq |G|$.
\end{fact}

\begin{fact}\label{fact:murthy}
\cite[Thm.~3.5]{HedtkeMurthy2011}
If $(S,T,U)$ is a TPP triple of subgroups of $G$ and one of $S$, $T$ or $U$ is normal in $G$, then $|S|\cdot |T|\cdot |U| \leq |G|$.
\end{fact}

\section{New Upper Bound for the TPP subgroup capacity}
\enlargethispage{\baselineskip}

\begin{thm} Let $G$ be a finite group. Let $\{S_i\}_{i=1}^{k}$ be the list of all subgroups of $G$, sorted by their order such that $|S_i|\leq |S_{i+1}|$. Let $N:=\max \{i : |S_i| \leq |G|/(|S_3|+|S_2|-1)\}$. We define
\begin{align}
\Delta(S_i) &:= \max \big\{|S_j|\cdot|S_k| ~ : ~ 1 < k < j < i, ~ |S_i|(|S_j|+|S_k|-1)\leq |G|\big\},\notag
\intertext{and}
b(G) &:= \max_{4 \leq i \leq N} \min\left\{ \frac{|G| \cdot |S_i|}{|\mathrm{Core}_G(S_i)|} ~ , ~ |S_i| \cdot\Delta(S_i) \right\}. \notag
\end{align}
Then \begin{gather}
\beta_\mathrm{g}(G) \leq \max\{b(G), |G|\} =: h(G).\label{eq:NewBound}
\end{gather}
\end{thm}

\begin{proof}
According to Fact~\ref{fact:order} we are only interested in triples of type $(S_i,S_j,S_k)$ where $i \geq j \geq k$. We assume that $|S_k|>1$, because a TPP triple $(S_i,S_j,S_k)$ represents a $(|S_i|\times |S_j|) \times (|S_j| \times |S_k|)$ matrix multiplication and we only focus on true matrix-matrix products. Furthermore $i\neq j \neq k$ holds, because in every other case the triple $(S_i,S_j,S_k)$ can not fulfill the TPP. Therefore it follows that $i\geq 4$.
Now assume that $(S_i,S_j,S_k)$ is a TPP triple in $G$. From Fact~\ref{fact:neumann} we know that $|S_i|(|S_j|+|S_k|-1)\leq |G|$ must hold. In the case where $S_j$ and $S_k$ are the smallest nontrivial distinct subgroups of $G$, what means that $j=3$ and $k=2$, this gives us the upper bound
\[
|S_i| \leq \frac{|G|}{|S_3|+|S_2|-1}
\]
for $S_i$. It follows that we can restrict the search space for $S_i$ to $\{S_i:  i \leq N\}$. Combined we get $4\leq i \leq N$.
From \textsc{Neumann} (Fact~\ref{fact:neumann}) we know the upper bound
\begin{gather*}
t(G) := \max \big\{|S_i|\cdot |S_j| \cdot |S_k|  ~~ : ~~  S_i,S_j,S_k < G, ~ |S_i| \geq |S_j| \geq |S_k| > 1, ~ |S_i|(|S_j|+|S_k|-1)\leq |G|\big\}
\end{gather*}
for $\beta_\mathrm{g}$. Note that this equals to
$
\max_{i} |S_i|\cdot \Delta(S_i)
$, the right-hand-side of $b(G)$.
Assume that $(S_i,S_j,S_k)$ is a TPP triple of subgroups of $G$. For every subset $A\subseteq S_i$ of $S_i$, $(A,S_j,S_k)$ is a TPP triple, too. If $S_i$ contains a normal subgroup $N\lhd G$ of $G$, then $(N,S_j,S_k)$ is a TPP triple of $G$ which fulfills the Fact~\ref{fact:murthy}. It follows that $|S_j|\cdot |S_k|\leq |G|/|N|$. Obviously, this holds for the biggest normal subgroup in $S_i$, too:
\[
|S_i|\cdot |S_j| \cdot |S_k| \leq |S_i| \cdot \frac{|G|}{|\mathrm{Core}_G(S_i)|}.
\]
Note that this is the left-hand-side of $b(G)$.
We omitted the case $(G,1,1)$, so it could be possible that $b(G) < |G|$. We correct this via Eq.~\eqref{eq:NewBound}.
\end{proof}

\section{Applications}

\noindent Our new bound $h$ is a combination of \textsc{Neumanns}'s bound $t$ (which is the formerly best known bound) and the observation about normal subgroups from \textsc{Hedtke} and \textsc{Murthy}. Obviously $\beta_g \leq h \leq t$ holds.
Note, that Eq.~\eqref{omega} leads to a nontrivial upper bound iff $\beta(G)>D_3(G)$.
Therefore we conclude that a group $G$ with $\beta_\mathrm{g}(G) \leq D_3(G)$ will never realize a nontrivial $\omega$ via a TPP triple of subgroups.
Tbl.~\ref{fig} shows the effect of $h$ and $t$ at excluding such $G$'s.

\begin{table}[h!]
\begin{tabular}{|l|r|rr|}
\hline
$|G|$ & \#$G$'s & $|\{G: t\leq D_3\}|$ & $|\{G: h\leq D_3\}|$\\
\hline
24& 12& 4& 6\\
32& 44& 7& 11\\
36& 10& 4& 6\\
40& 11& 3& 11\\
48& 47& 18& 22\\
50& 3& 1& 2\\
56& 10& 2& 4\\
60& 11& 5& 8\\
\hline
\end{tabular}
\begin{tabular}{|l|r|rr|}
\hline
$|G|$ & \#$G$'s & $|\{G: t\leq D_3\}|$ & $|\{G: h\leq D_3\}|$\\
\hline
64& 256& 129& 136\\
72& 44& 8& 12\\
80& 47& 18& 22\\
84& 13& 5& 8\\
88& 9& 0& 2\\
96& 224& 28& 93\\
98& 3& 1& 2\\
100& 12& 6& 8\\
\hline
\end{tabular}
\caption{Examples of the impact of the new bound for \emph{nonabelian} groups.}
\label{fig}
\end{table}

\providecommand{\bysame}{\leavevmode\hbox to3em{\hrulefill}\thinspace}
\providecommand{\MR}{\relax\ifhmode\unskip\space\fi MR }
\providecommand{\MRhref}[2]{%
  \href{http://www.ams.org/mathscinet-getitem?mr=#1}{#2}
}
\providecommand{\href}[2]{#2}

\end{document}